\documentclass[11pt,reqno]{amsart}
\usepackage{amssymb}
\usepackage{amsmath}
\usepackage{amsthm}
\usepackage{color}
\makeatletter
\def\LaTeX{\leavevmode L\raise.42ex
    \hbox{\kern-.3em\size{\sf@size}{0pt}\selectfont A}\kern-.15em\TeX}
\makeatother

\newcommand{\BibTeX}{{\rm B\kern-.05em{\sc
          i\kern-.025emb}\kern-.08em\TeX}}

\makeatletter
\def\@currentlabel{2.1}\label{e:dispaa}
\def\@currentlabel{2.21}\label{e:dispau}
\def\@currentlabel{2.22}\label{e:dispav}
\def\@currentlabel{2.23}\label{e:dispaw}
\def\@currentlabel{2.24}\label{e:dispax}
\def\theequation{\thesection.\@arabic\c@equation}
\makeatother

\newcounter{mnotecount}[section]

\newcommand{\rmnote}[1]{}

\renewcommand{\theequation}{\arabic{section}.\arabic{equation}}
\newtheorem{thm}{Theorem}[section]

\newtheorem{cor}[thm]{Corollary}

\newtheorem{rem}[thm]{Remark}

\newtheorem{lemma}[thm]{Lemma}

\newcommand{\R}{\mathbb{R}}
\newcommand{\Z}{\mathbb{Z}}

\begin{document}
\title[the Hilbert transform along variable non-flat curves]{$L^p$ bound for the Hilbert transform along variable non-flat curves}

\author[R. Wan]{ Renhui Wan}

\address{Institute of Mathematics, School of Mathematical Sciences, Nanjing Normal University, Nanjing 210023, China}

\email{rhwanmath@163.com}

 \date{\today}
\numberwithin{equation}{section}
 \maketitle
\begin{abstract}
We prove the $L^p$ bound for the  Hilbert transform along variable non-flat curves $(t,u(x)[t]^\alpha+v(x)[t]^\beta)$, where $\alpha$ and $\beta$ satisfy
 $\alpha\neq \beta,\ \alpha\neq 1,\ \beta\neq 1.$
 Comparing  with the associated theorem in \cite{GHLJ} investigating the case $\alpha=\beta\neq 1$, our result is more general while the proof is more involved. To achieve our goal,   we divide the frequency of the objective function  into three cases and  take different strategies to control these cases. Furthermore, we need to introduce a ``short" shift maximal function $\mathfrak{M}^{[n]}$ to establish some pointwise estimate.
\end{abstract}
\maketitle
\vskip .3in
\section{Introduction}
\label{s1}
\vskip .1in
Let $u:\R\times\R\rightarrow \R$ be a measurable function, the Hilbert transform along  variable curves $(t,u(x,y)\cdot [t]^\alpha)$ is defined by
\begin{equation*}
\mathcal{H}_{\alpha}f (x,y):=p.v.\int_{\R} f(x-t,y-u(x,y)[t]^\alpha)  \frac{dt}{t},\ \alpha>0,\
\end{equation*}
where $[t]^\alpha$ represents  either $|t|^\alpha$ or sgn$(t)|t|^\alpha$ for  $\alpha\neq 1$ and   $[t]^1$ stands for $t$.
The research of $\mathcal{H}_{\alpha}$ is considered as  a complement of the works on the $L^p$ boundedness of the maximal operator along  variable curves $(t,u(x,y)\cdot [t]^\alpha)$ given by
$$\mathcal{M}_\alpha f(x,y):=\sup_{\epsilon>0}\frac{1}{2\epsilon}
\int_{-\epsilon}^\epsilon\left|f(x-t,y-u(x,y)[t]^\alpha)
\right|dt.$$
We refer \cite{B89}, \cite{CNSW99} and reference therein for the investigation of the above maximal function.
\vskip.1in
Stein and Street \cite{SS12} proved that $\mathcal{H}_{\alpha} $ is bounded on $L^p$ for $p>1$ under the assumption that $u\in \mathbb{N}$ is analytic. Indeed, the objects in that work is all polynomials with analytic coefficients. Note that this analytic condition for $u$  plays a crucial role in their proof.
Based on time-frequency techniques, Lacey and Li \cite{LL06} removed this type of  regularity assumption. More precisely,   they achieved  the following two single annulus estimates for an arbitrary measurable function $u$:
$$\|\mathcal{H}_{1}P_k^{(2)}f\|_{2,\infty}\lesssim\  \|P_k^{(2)}f \|_2,$$
where $\|f\|_{2,\infty}$ means the weak $L^2$ norm of $f$,
and
\begin{equation*}
\|\mathcal{H}_{1}P_k^{(2)}f\|_{p}\lesssim\  \|P_k^{(2)}f\|_p,\ p>2.
\end{equation*}
We refer (\ref{jiajia}) for the definition of $P_k^{(2)}$.
For some convenience, in what follows, we call single annulus estimate  weak estimate.
 In the case of $\alpha\neq 1$, by using the local smoothing estimate which is derived from a decoupling inequality (see \cite{W00}) and the Sobolev embedding inequality,    Guo, Hickman, Lie and Roos \cite{GHLJ} established
$$\|\mathcal{H}_{\alpha}P_k^{(2)}f\|_{p}\lesssim\  \|P_k^{(2)}f\|_p,\ p>2.
$$
However, all the above works only devote to the weak estimate. Lately, in two papers \cite{GRSP,GRSP2},  Guo,  Roos, Seeger and   Yung gave a detailed discussion of the  $L^p$ boundedness of $\mathcal{H}_\alpha$ for $\alpha>1$. More precisely, the $L^p$ operator norm of $\mathcal{H}_\alpha$ depends on the choice of $u(x,y)$.
\vskip.1in
There are many works which focus on some special $u$ satisfying
$$u(x,y)=u(x,0).$$
It is corresponding to
 the operator $H_\alpha$ defined by
$$H_{\alpha}f (x,y):=p.v.\int_{\R} f(x-t,y-u(x)[t]^\alpha)  \frac{dt}{t},\ \alpha>0,
$$
where $u:\R\rightarrow \R$ is a measurable function. This operator is very related to  the Carleson's maximal operator (see e.g. \cite{Car66,L20}) in the sense that their $L^2$ bounds are equivalent, which can be obtained via the Plancherel theorem and the linearization process. By further developing  the methods applied in \cite{LL06} and \cite{LL10},  Bateman \cite{B13}, Bateman and Thiele \cite{BC}  considered the operator $H_1$ and
proved
$$\|H_{1}P_k^{(2)}f\|_{p}\lesssim\  \|P_k^{(2)}f\|_p,\ \ \ 1<p<\infty
$$
and
$$
\|H_{1}f\|_{p}\lesssim\  \|f\|_p,\ \ \ \frac{3}{2}<p<\infty.
$$
For $\alpha\neq 1$,  Guo, Hickman, Lie and Roos \cite{GHLJ} proved
$$\|H_{\alpha}f\|_{p}\lesssim\  \|f\|_p,\ \ \ 1<p<\infty$$
by Littelwood-Paley theory and the shifted maximal estimate (see Lemma \ref{sme}). In addition,
 they also proved
$$\|H_{1,\alpha}P_k^{(2)}f\|_{p}\le C_p \|P_k^{(2)}f\|_p,\ 1<p<\infty,$$
where $H_{1,\alpha}$ is given by
$$H_{1,\alpha}f (x,y):=p.v.\int_{\R} f(x-t,y-u(x)t-v(x)[t]^\alpha)  \frac{dt}{t},\ \alpha>0,\ \alpha\neq 2.$$
Their proof  relies  on  almost-orthogonality, stationary-phase and $TT^\star$ methods.
However, it is not easy to prove the $L^p$ bound for  $H_{1,\alpha}$.
 It is worth mentioning  that  the method in \cite{GHLJ} seems not useful    to the case $\alpha=1$ since it strongly depends on the curvature condition coming from $\alpha\neq 1$. In turn, one scarcely obtains the full range of $p$ when applying the method in \cite{B13,BC} to the case $\alpha\neq 1$.
\vskip.1in
Motivated by the above works, here we consider the operator  $H_{\alpha,\beta}$ defined by
\begin{equation}\label{n1}
H_{\alpha,\beta}f (x,y):=p.v.\int_{\R} f(x-t,y-u(x)[t]^\alpha-v(x)[t]^\beta)  \frac{dt}{t},\ \alpha,\beta>0.
\end{equation}
where $u:\R\rightarrow \R$ and $v:\R\rightarrow\R$ are two measurable functions.  The natural goal is to prove the $L^p$ boundedness of this operator.
Using the notation in (\ref{n1}), we collect the previous related works as follow:
\vskip.2in
\begin{tabular}{|l|l|l|}
  \hline
  $\alpha=\beta=1$ &  \cite{B13,BC} &  $\frac{3}{2}<p<\infty$ and weak estimate $1<p<\infty$ \\
  \hline
 $\alpha=\beta\neq1$ & \cite{GHLJ}  & $1<p<\infty$\\
  \hline
  $\alpha\neq \beta, \alpha=1, \beta\neq 2

  $  &   \cite{GHLJ} & weak estimate $1<p<\infty$  \\
  \hline
\end{tabular}
\vskip.1in
Now, we state our main result.
\begin{thm}\label{t2}
Let $(\alpha,\beta)\in S_1$ defined by
\begin{equation*}
S_1:=\big\{(\alpha,\beta)\in \R^+\times \R^+:\ \alpha\neq \beta,\ \alpha\neq 1,\ \beta\neq 1\big\},
\end{equation*}
then we have
\begin{equation}\label{aim}
\|H_{\alpha,\beta}f\|_p\lesssim_{\alpha,\beta,p} \|f\|_p
\end{equation}
holds for $1<p<\infty$.
\end{thm}
\begin{rem}\label{r1}
 As far as we know, it is unsolved for the followng two cases:
$$S_2:=\big\{(\alpha,\beta)\in \R^+\times \R^+:\ \alpha\neq \beta,\ \alpha=1,\ \beta\neq 2\big\},\ S_3:=\{(1,2)\}.$$
We give the following remarks on the case $p\neq 2$ since the special case $p=2$ is a direct result of the  Carleson maximal estimate.\\
(1) In the case of $(\alpha,\beta)\in S_2$, it  seems hard to directly use
our approach, since $u(x)t$ does not have any curvature such that it is not easy to obtain  a useful estimate like Lemma \ref{l1.1}. We speculate that its proof needs a combination of the methods in previous works and our approach in the current paper. For all this, it remains open.\\
(2) For $(\alpha,\beta)\in S_3$, as the Carleson maximal operator, it is the most natural challenge. We are far from knowing how to bound this case.
\end{rem}
Next, let us make some comments on our proof.
\vskip.1in
We first show the new gap in the proof and then give our strategy.
Since there are  two fractional variable monomial
$u(x)[t]^\alpha$ and $v(x)[t]^\beta$,
we need  a different shifted maximal function $M^{[n]}$ (see section \ref{s2} for the definition) to control some pointwise estimate, where $n$ depends on the frequncy of the objective function (This fact makes our proof and the correspongding proof in  \cite{GHLJ} different).  More badly, this type of shifted maximal function  prevents us from   using  the vector-valued shifted maximal estimate. To break this barrier, we first divide the frequency of the objective function into three cases, and  use different measures to deal with these cases.
Then we introduce a ``short" shift maximal function $\mathfrak{M}^{[n]}$ to obtain some pointwise estimate.  At last,  we can prove the desired estimate by applying both the vector-valued and the scalar shifted maximal estimates.
\vskip.2in
The following  weak estimate is a direct result of Theorem \ref{t2} by taking $f\rightarrow P_k^{(2)}f$.
\begin{cor}[weak estimate]
Let $(\alpha,\beta)\in S_1$, then for all $k\in\Z$
\begin{equation*}
\|H_{\alpha,\beta}P_k^{(2)}f\|_p\lesssim_{\alpha,\beta,p} \|P_k^{(2)}f\|_p.
\end{equation*}
 holds for $1<p<\infty$.
\end{cor}
\vskip .1in
The present paper is structured as follows:\\
 In Section \ref{s2},  we give identity decomposition, the  Littlewood-Paley projection, and some useful estimates such as the shifted maximal estimate.  The third section proves Theorem \ref{t2}, while the followed section  gives the proof of Lemma  \ref{4.1}. Finally, we give the proof of Lemma  \ref{l2} in the last section.
 \vskip.1in

Let us complete this section by describing the notation we shall use in this paper.\\
{\bf Notation.}  We hereinafter  use $x\lesssim y$ to stand for there exists a constant $C$ ( which may only depend on fixed parameters such as $\alpha$, $\beta$ and $p$ ) such that $x\le Cy$. We write $C_{\gamma}$ to mean that the constant $C$ depends on $\gamma$.
$\mathcal{F}^y(f)$ is the Fourier transform in the $y-$variable of a function $f$. We use
$\|\cdot\|_p$ to stand for $\|\cdot\|_{L^p}$.
\vskip .3in
\section{Some preparations}
\label{s2}
\vskip .1in
Let $\phi(t)$ be a radial, smooth and decreasing function which is supported on $\{|t|\le 2\}$ and equals 1 on $\{|t|\le1\}$. Define $\psi(t) $ by
$\psi(t)=\phi(t)-\phi(2t)$, which is a non-negative smooth function supported on the set $\{t:1/2\le |t|\le 2\}$. Denote $\psi_l(t):=\psi(2^{-l}t)$, $\phi_l(t):=\phi(2^{-l}t)$ (Note that $\phi_0(t)=\phi(t)$), then for all $l_0\in\Z$,
\begin{equation}\label{decom}
\phi_{l_0}(t)+\sum_{l> l_0} \psi_l(t)=1,\ \forall\ t\in\R
\end{equation}
and
$$\sum_{l\in\Z}\psi_l(t)=1,\ \forall\ t\in\R\setminus\{0\}.$$
Define the corresponding Littlewood-Paley projection in the $y-$variable  of a function $f$ on $\R$ by
\begin{equation}\label{jiajia}
P_k^{(2)} f(x,y):=\int_{\R} f(x,y-z)\check{\psi_k}(z)dz,
\end{equation}
where $\check{\psi_k}$ denotes the inverse Fourier transform of the function $\psi_k$.
\vskip.1in
We need the following decay estimate, which is a modification of Lemma 2.1 in \cite{GPRY}.
\begin{lemma}\label{l1.1}
Let  $\alpha,\beta$ be two positive numbers satisfying $\alpha\neq \beta$, $\alpha,\beta\neq 1$, and $\psi$ be smooth and supported on $[1/2,2]\cup [-2,-1/2]$ and $\theta$ be smooth and supported on $[-2,2]$. For $\lambda=(\lambda_1,\lambda_2)\in\R^2$, $b\in \R$ and $a>0$, let
\begin{equation}\label{para}
\Phi^{\lambda,b}(t)=e^{i\lambda_1[t]^\alpha+
i\lambda_2[t]^\beta}\frac{\theta(bt)\psi(t)}{t},\ \Phi^{\lambda,b}_a(t)=\frac{1}{a}\Phi^{\lambda,b
}(\frac{t}{a}).
\end{equation}
Then there exists  $\gamma_0>0$ such that for all $r\ge 1$ and $\digamma\in L^2(\R)$,
$$\big\|\sup_{\begin{array}{c}
a>0,b\in \R \\
|\lambda|:=|\lambda_1|+|\lambda_2|\ge r
\end{array}}
|\big(\digamma\star\Phi^{\lambda,b}_a\big) (x) |\big\|_{L^2(dx)}\lesssim_{\alpha,\beta} r^{-\gamma_0}\|\digamma\|_{L^2(dx)}.$$
\end{lemma}
\begin{proof}
Due to the supports of $\theta$ and $\psi$,
the effective range for $b$ is  $\{|b|\le 4\}$. Then we can replace $b\in \R$ by $|b|\le 4$. Following the proof of Lemma 2.1 in \cite{GPRY} line by line can lead to the desired estimate.
\end{proof}
Suppose $\sigma\ge 0$, we define the shifted maximal operator $M^{[\sigma]}$ and ``short" shifted maximal operator $\mathfrak{M}^{[\sigma]}$ as follows:
$$M^{[\sigma]}f (z):=\sup_{z\in I\subset \R}\frac{1}{|I|}
\int_{I^{(\sigma)}}|f(\xi)|d\xi$$
and
$$\mathfrak{M}^{[\sigma]}f (z):=\sup_{z\in I\subset \R,|I|=1}
\int_{I^{(\sigma)}}|f(\xi)|d\xi,$$
where $I^{(\sigma)}$ denotes a shift of the bounded interval $I=[a,b]$ given by
$$I^{(\sigma)}:=[a-\sigma|I|,b-\sigma|I|]\cup [a+\sigma|I|,b+\sigma|I|].$$
Obviously, $\mathfrak{M}^{[\sigma]}f (z)\le\ M^{[\sigma]}f (z)$.
\vskip.1in
Our proof needs the following estimates.
\begin{lemma}\cite{M14,S93}\label{ssme}
Let $1<p<\infty$,
we have
\begin{equation}\label{ssme1}
\|M^{[n]}f\|_p
\lesssim\ \log (2+|n|) \|f\|_p.
\end{equation}
Here the constant hidden in $``\lesssim"$ is independent of $|n|$ and $f$.
\end{lemma}

\begin{lemma}\cite{GHLJ}\label{sme}
Let $1<p<\infty$, $1<q\le \infty$,
we have
\begin{equation}\label{sme1}
\Big\|\big(\sum_{k\in\Z}|M^{[n]}f_k|^q\big)^\frac{1}{q}\Big\|_p
\lesssim\ \log^2(2+|n|) \Big\|\big(\sum_{k\in\Z}|f_k|^q\big)^\frac{1}{q}\Big\|_p.
\end{equation}
Here the constant hidden in $``\lesssim"$ is independent of $|n|$ and $f$.
\end{lemma}
We also require the following lemma to prove a pointwise estimate.
\begin{lemma}\label{lnew}
Let $n=n_1+n_2$, $n_1\ge 0$ and $0\le n_2\le C_1$ for some positive constant $C_1$, then
the following pointwise estimate of $\mathfrak{M}^{[n]}$ holds:
$$\mathfrak{M}^{[n]}f\le\ \sum_{l=0}^{C_1}\mathfrak{M}^{[n_1+l]}f.$$
\end{lemma}
\begin{proof}
Let $I=[a,a+1]$, then
$$
\begin{aligned}
\int_{I^{(n)}}|f(\xi)|d\xi
=&\int_{a+n}^{a+1+n}|f(\xi)|d\xi+\int_{a-n}^{a+1-n}|f(\xi)|d\xi\\
=&\int_{a+n_1+n_2}^{a+1+n_1+n_2}|f(\xi)|d\xi
+\int_{a-n_1-n_2}^{a+1-n_1-n_2}|f(\xi)|d\xi.
\end{aligned}
$$
Since $0\le n_2\le C_1$, we have
$$\int_{a+n_1+n_2}^{a+1+n_1+n_2}|f(\xi)|d\xi
\le\ \sum_{l=0}^{C_1}\int_{a+n_1+l}^{a+1+n_1+l}|f(\xi)|d\xi$$
and
$$
\int_{a-n_1-n_2}^{a+1-n_1-n_2}|f(\xi)|d\xi
\le \sum_{l=0}^{C_1}\int_{a-n_1-l}^{a+1-n_1-l}|f(\xi)|d\xi.
$$
Combining with these inequalities yields
$$\int_{I^{(n)}}|f(\xi)|d\xi
\le\ \sum_{l=0}^{C_1}\int_{I^{(n_1+l)}}|f(\xi)|d\xi,$$
which completes the proof by  taking the supremum over $a\in\R$ on the both sides.
\end{proof}
\vskip .3in
\section{Proof of Theorem \ref{t2}}
\label{s4}
\vskip .1in
In this section, we prove Theorem \ref{t2}. Without loss of generality, we assume $u(x)>0$, $v(x)>0$ and $0<\alpha<\beta$. For convenience of notation, we denote $u(x)[t]^\alpha+v(x)[t]^\beta$ by $\Gamma(x,t)$.
\vskip.1in
Thanks to
$$P_k^{(2)}H_{\alpha,\beta}f=H_{\alpha,\beta}P_k^{(2)}f$$
and the Littlewood-Paley theory, it is enough to show
\begin{equation}\label{4.1}
\Big\|\big(\sum_{k\in\Z}|H_{\alpha,\beta}P_k^{(2)}f
|^2\big)^\frac{1}{2}\Big\|_p\lesssim\ \|f\|_{p}.
\end{equation}
\vskip.1in
We beforehand give some motivations of the strategy.  We first use the identity decomposition like (\ref{decom}) to quantify $\Gamma(x,t)$  since $H_{\alpha,\beta}$ is similar to the classical Hilbert transform for ``small" $|\Gamma(x,t)|$. After this process,  we  establish  a useful decay estimate like $2^{-\iota n}$ ($\iota>0$)
for $|\Gamma(x,t)|\thickapprox 2^n$, and then obtain the desired estimate for not ``small" $|\Gamma(x,t)|$ by summing over $n\ge 0$.
The above approach is analogous  to that in \cite{GHLJ},
however, due to the appearance of  $u(x)[t]^\alpha$ and $v(x)[t]^\beta$, our proof is more involved. On the one hand, on account of $\alpha\neq \beta$, we need to employ
\begin{equation}\label{d2}
\phi_{-k/\alpha}(u(x)^\frac{1}{\alpha})+\sum_{l> -k/\alpha} \psi_l(u(x)^\frac{1}{\alpha})=1
\end{equation}
and
\begin{equation}\label{d3}
\phi_{-k/\beta}(v(x)^\frac{1}{\beta})+\sum_{l> -k/\beta} \psi_l(v(x)^\frac{1}{\beta})=1
\end{equation}
to
quantify $u(x)[t]^\alpha$ and $v(x)[t]^\beta$, respectively.
On the other hand, to bound some operators by shifted maximal operator,  we have to classify the range of $k$ in (\ref{4.1})  into several small ranges, and use several different measures to control them.
 \vskip.1in
 We return to the proof.
Applying (\ref{d2}) and (\ref{d3}) to $H_{\alpha,\beta}P_k^{(2)}f$ gives
$$
\begin{aligned}
H_{\alpha,\beta}P_k^{(2)}f(x,y)=&
H_{\alpha,\beta}^{ll}P_k^{(2)}f(x,y)+
H_{\alpha,\beta}^{lh}P_k^{(2)}f(x,y)\\
&+H_{
\alpha,\beta}^{hl}P_k^{(2)}f(x,y)
+H_{\alpha,\beta}^{hh}P_k^{(2)}f(x,y),
\end{aligned}
$$
where
$$
H_{\alpha,\beta}^{ll}P_k^{(2)}f(x,y)
=p.v.\int_\R (P_k^{(2)}f)(x-t,y-\Gamma(x,t))
\phi_{-k/\alpha}(u(x)^\frac{1}{\alpha}t)
\phi_{-k/\beta}(v(x)^\frac{1}{\beta}t)\frac{dt}{t},\\
$$

$$
H_{\alpha,\beta}^{lh}P_k^{(2)}f(x,y)
=\sum_{l> -k/\beta}\int_\R (P_k^{(2)}f)(x-t,y-\Gamma(x,t))
\phi_{-k/\alpha}(u(x)^\frac{1}{\alpha}t)
\psi_{l}(v(x)^\frac{1}{\beta}t)\frac{dt}{t},\\
$$

$$
H_{\alpha,\beta}^{hl}P_k^{(2)}f(x,y)
=\sum_{l> -k/\alpha}\int_\R (P_k^{(2)}f)(x-t,y-\Gamma(x,t))
\psi_{l}(u(x)^\frac{1}{\alpha}t)
\phi_{-k/\beta}(v(x)^\frac{1}{\beta}t)\frac{dt}{t},\\
$$

$$
\begin{aligned}
H_{\alpha,\beta}^{hh}P_k^{(2)}f(x,y)
=&\sum_{\begin{array}{c}
          l> -k/\beta, \\
          m> -k/\alpha
        \end{array}
}\int_\R (P_k^{(2)}f)(x-t,y-\Gamma(x,t))
\psi_{m}(u(x)^\frac{1}{\alpha}t)
\psi_{l}(v(x)^\frac{1}{\beta}t)\frac{dt}{t}.
\end{aligned}
$$
 In order to prove (\ref{4.1}), it suffices to
 show (\ref{4.1}) with
$H_{\alpha,\beta}$ replaced by $H_{\alpha,\beta}^{ll}$,
$H_{\alpha,\beta}^{lh}$, $H_{\alpha,\beta}^{h l}$ and
$H_{\alpha,\beta}^{hh}$, respectively. We stress that the new ingredient in our proof is the estimate of $H_{\alpha,\beta}^{hh}P_k^{(2)}f$.
\subsection{The estimate of $H_{\alpha,\beta}^{ll}P_k^{(2)}f$}
We need the following lemma, the proof of which is given in Section \ref{s5}.
\vskip.1in
 Let $\mathcal{M}^{(i)}f(x,y) (i=1,2)$ be the strong maximal function applied in the $i$-th variable, $H^\star$ be the maximally truncated Hilbert transform applied in the first variable.
\begin{lemma}\label{l4.1}
The following pointwise estimate holds
\begin{equation}\label{low}
\begin{aligned}
H_{\alpha,\beta}^{ll}P_k^{(2)}f(x,y)
\lesssim&\ \sum_{q\in\Z}\frac{1}{(1+|q|)^2}
\int_q^{q+1} \mathcal{M}^{(1)}P_k^{(2)}f(x,y-z)  dz\\
&+\mathcal{M}^{(1)}P_k^{(2)}f(x,y)+H^\star P_k^{(2)}f (x,y),
\end{aligned}
\end{equation}
where the constant is independent of $u(x)$, $v(x)$, $k$ and $f$.
\end{lemma}
Thanks to  Lemma \ref{l4.1}, the estimate of $H_{\alpha,\beta}^{ll}P_k^{(2)}f$ is a direct consequence. In fact, by  Minkowski's inequality, we bound the corresponding norm of the first term on the right side of (\ref{low}) by
$$
\begin{aligned}
\sum_{q\in\Z}\frac{1}{(1+|q|)^2}
\int_q^{q+1}\Big\| \big(\sum_{k\in\Z}
|\mathcal{M}^{(1)}P_k^{(2)}f(x,y-z)|^2\big)^\frac{1}{2}\Big\|_p  dz
\lesssim&\ \Big\| \big(\sum_{k\in\Z}
|\mathcal{M}^{(1)}P_k^{(2)}f|^2\big)^\frac{1}{2}\Big\|_p.
\end{aligned}
$$
As a result, we  achieve (\ref{4.1}) with $H_{\alpha,\beta}$ replaced by $H_{\alpha,\beta}^{ll}$ via (\ref{low}) and the vector estimates of $\mathcal{M}^{(1)}$ and $H^\star$.
\subsection{The estimate of $H_{\alpha,\beta}^{hh}P_k^{(2)}f$}
 In fact, the estimate of $H_{\alpha,\beta}^{hh}P_k^{(2)}f$ is the core content in this article. Changing variables $m\rightarrow m-\frac{k}{\alpha}$ and  $l\rightarrow j-\frac{k}{\beta}$, we have
 $$H_{\alpha,\beta}^{hh}P_k^{(2)}f=\sum_{j,m\ge 0}
 T_{k,m,j}(P_k^{(2)}f)(x,y),$$
 where $T_{k,m,j}$ is defined by
 $$T_{k,m,j} f(x,y)
 =\int_\R f(x-t,y-u(x)[t]^\alpha-v(x)[t]^\beta)
\psi_{m-\frac{k}{\alpha}}(u(x)^\frac{1}{\alpha}t)
\psi_{j-\frac{k}{\beta}}(v(x)^\frac{1}{\beta}t)\frac{dt}{t}.$$
 To get the desired estimate of $H_{\alpha,\beta}^{hh}P_k^{(2)}f$,
 it suffices to show  there exists $\tau_1>0$ such that
 \begin{equation}\label{4.2}
 \Big\|\big(\sum_{k\in\Z}|T_{k,m,j}(P_k^{(2)}f)|^2
 \big)^\frac{1}{2}\Big\|_p
 \lesssim\ 2^{-\max\{j\beta,m\alpha\}\tau_1}\|f\|_p.
 \end{equation}
  By   Stein-Wainger's method \cite{SW01}, we first prove the special case $p=2$, which is described by  the following lemma. We postpone  the proof  in Section \ref{s6}.
 \begin{lemma}\label{l2}
 (\ref{4.2}) holds for $p=2$.
 \end{lemma}
 Thanks to interpolation theorem, it is enough to prove
 \begin{equation}\label{4.3}
 \Big\|\big(\sum_{k\in\Z}|T_{k,m,j}(P_k^{(2)}f)|^2
 \big)^\frac{1}{2}\Big\|_p
 \lesssim\ (j^2\beta^2+m^2\alpha^2 +C_{\alpha,\beta})^2\|f\|_p,
 \end{equation}
 where $C_{\alpha,\beta}$ is a constant which only depends on $\alpha$ and $\beta$.
As the statement at the beginning of this section, we will split the sum of $k$ into three   parts
$$\sum_{k\in\Z}=\sum_{k< k_{\alpha,\beta,x}^1}+\sum_{ k_{\alpha,\beta,x}^1\le k< k_{\alpha,\beta,x}^2}+\sum_{k\ge k_{\alpha,\beta,x}^2},$$
where
$$k_{\alpha,\beta,x}^1=\frac{ \log_2 \big(v(x)^\alpha u(x)^{-\beta}\big)-\alpha\beta j}{\beta-\alpha},\
\ k_{\alpha,\beta,x}^2=\frac{\alpha\beta m+ \log_2 \big(v(x)^\alpha u(x)^{-\beta}\big)}{\beta-\alpha}.$$
Note that
\begin{equation}\label{jia1}
k_{\alpha,\beta,x}^2-k_{\alpha,\beta,x}^1=\frac{\alpha\beta(m+j)
}{\beta-\alpha}\ge0.
\end{equation}
\vskip.1in
To prove (\ref{4.3}), it suffices to prove
\begin{equation}\label{a1}
\begin{aligned}
\Big\|\big(\sum_{k< k_{\alpha,\beta,x}^1}|T_{k,m,j}(P_k^{(2)}f)|^2
 \big)^\frac{1}{2}\Big\|_p
 \lesssim\ (2+j^2\beta^2+m^2\alpha^2)\|f\|_p
\end{aligned}
\end{equation}

\begin{equation}\label{a2}
\begin{aligned}
\Big\|\big(\sum_{k\ge k_{\alpha,\beta,x}^2}|T_{k,m,j}(P_k^{(2)}f)|^2
 \big)^\frac{1}{2}\Big\|_p
 \lesssim\ (2+j^2\beta^2+m^2\alpha^2)\|f\|_p
\end{aligned}
\end{equation}

\begin{equation}\label{a3}
\begin{aligned}
\Big\|\big(\sum_{ k_{\alpha,\beta,x}^1\le k< k_{\alpha,\beta,x}^2 }|T_{k,m,j}(P_k^{(2)}f)|^2
 \big)^\frac{1}{2}\Big\|_p
 \lesssim\ (j^2\beta^2+m^2\alpha^2+C_{\alpha,\beta})^2\|f\|_p
\end{aligned}
\end{equation}
 Denote
 $$F_k(x,y):=f(x,2^{-k}y),\ \ (u_k(x),v_k(x))=2^k(u(x),v(x)).$$
Changing variables $z\rightarrow 2^{-k}z$ gives
 $$
 \begin{aligned}
 &\ \ (P_k^{(2)}f)(x-t,y-u(x)[t]^\alpha-v(x)[t]^\beta)\\
 =&\int 2^k\breve{\psi}(2^kz) f(x-t,y-z-u(x)[t]^\alpha-v(x)[t]^\beta)dz\\
 =& \int \breve{\psi}(z) f(x-t,y-2^{-k}z-u(x)[t]^\alpha-v(x)[t]^\beta)dz\\
 =& \int \breve{\psi}(z) f\big(x-t, 2^{-k}(2^ky-z-u_k(x)[t]^\alpha-v_k(x)[t]^\beta)\big)dz\\
 =& (P_0^{(2)}F_k)(x-t,2^ky-u_k(x)[t]^\alpha-v_k(x)[t]^\beta).
 \end{aligned}
 $$
Thus, we rewrite   $T_{k,m,j}(P_k^{(2)}f)(x,y)$ as
 \begin{equation*}
   \begin{aligned}
 \int_\R (P_0^{(2)}F_k)(x-t,2^ky-u_k(x)[t]^\alpha-v_k(x)[t]^\beta)
 \psi_m(u_k(x)^\frac{1}{\alpha}t)\psi_j(v_k(x
 )^\frac{1}{\beta}t)\frac{dt}{t}
   \end{aligned}
 \end{equation*}
 Next, we focus on the pointwise estimate  of the following integral:
 \begin{equation}\label{ii}
 TP_0^{(2)}f(x,y)=\int_\R (P_0^{(2)}f)(x-t,y-u_k(x)[t]^\alpha-v_k(x)[t]^\beta)
 \psi_m(u_k(x)^\frac{1}{\alpha}t)\psi_j(v_k(x
 )^\frac{1}{\beta}t)\frac{dt}{t},
 \end{equation}
 which can yield the associated estimate of  $T_{k,m,j}(P_k^{(2)}f)$ by the scaling arguments.
\vskip.1in
{\bf $\bullet$ The proof of (\ref{a1})}
 Due to the support  of $\psi$,
 \begin{equation}\label{yong}
 \lambda_{x,j}:=2^jv_k(x)^{-\frac{1}{\beta}}\in [ 2^{m-1}u_k(x)^{-\frac{1}{\alpha}},2^{m+1}u_k(x)^{-\frac{1}{\alpha}}]
 \end{equation}
 and
 the range of $t$ in (\ref{ii}) is
 $$
 \begin{aligned}
 \mathcal{I}:&=\{t:2^{j-1}v_k(x)^{-\frac{1}{\beta}}\le |t|\le 2^{j+1}v_k(x)^{-\frac{1}{\beta}}\}\\
 &\cap\{t:2^{m-1}u_k(x)^{-\frac{1}{\alpha}}
  \le |t|\le 2^{m+1}u_k(x)^{-\frac{1}{\alpha}}\}.
  \end{aligned}$$
 Because of
 $$P_0^{(2)}f(x,y)=\int \breve{\psi}(z)f (x,y-z)dz$$
 yielding
 $$
 \begin{aligned}
 |P_0^{(2)}f(x,y)|
 \lesssim&\  \int \frac{1}{1+|z|^8}|f (x,y-z)|dz\\
 \lesssim&\  \int \sum_{\tau\in \Z}\chi_{[\tau,\tau+1)}(z)\frac{1}{1+|z|^8}|f (x,y-z)|dz\\
 \lesssim&\ \sum_{\tau\in\Z}\frac{1}{1+|\tau|^8}
 \int_\tau^{\tau+1} |f (x,y-z)|dz,
 \end{aligned}
 $$
 we bound $|TP_0^{(2)}f(x,y)|$ by
 $$
 \begin{aligned}
  &\ \  \sum_{\tau\in\Z}\frac{1}{1+|\tau|^8}
 \int \int_\tau^{\tau+1} |f (x-t,y-z-u_k(x)[t]^\alpha-v_k(x)[t]^\beta)|\\
&\  \hskip.8in\psi_m(u_k(x)^\frac{1}{\alpha}t)\psi_j(v_k(x
 )^\frac{1}{\beta}t)dz\frac{dt}{|t|}\\
 \lesssim&\ \ \sum_{\tau\in\Z}\frac{1}{1+|\tau|^8}\frac{1}{\lambda_{x,j}}
 \int_{t\in \mathcal{I}} \int_\tau^{\tau+1} |f (x-t,y-z-u_k(x)[t]^\alpha-v_k(x)[t]^\beta)|\\
 &\hskip.9in
\psi_j(v_k(x
 )^\frac{1}{\beta}t)dz dt.
 \end{aligned}
 $$
 The following procedure is  splitting $\mathcal{I}$ into small intervals $\{I_r\}$, which makes $u(x)[t]^\alpha\in I_{x,t}^1$  and $v(x)[t]^\beta\in I_{x,t}^2$ (here $|I_{x,t}^i|\lesssim 1 (i=1,2)$), and estimating the above integral by  the shifted maximal function.
 \vskip.1in
 Choosing $\delta_{x,j,m}=\lambda_{x,j}\min\{2^{-m\alpha},2^{-j\beta}\}$
  (this choice of $\delta_{x,j}$ is to get (\ref{goal})). We define
  $$ I_{r}:=\{t:\ \frac{1}{2}\lambda_{x,j}
  +r\delta_{x,j,m}
  \le |t|\le \frac{1}{2}\lambda_{x,j}
  +(r+1)\delta_{x,j,m}\},$$
  and then we see $|I_r|=2\delta_{x,j,m}$ and
  $$\mathcal{I}\subset \{\frac{1}{2}\lambda_{x,j}\le|t|\le 2\lambda_{x,j}\}  \subset \big( \bigcup_{r=0}^{N_{j,m}-1} I_r\big),$$
  where $\frac{3}{2}\lambda_{x,j}\le N_{j,m}\delta_{x,j}\le 2\lambda_{x,j}$ (which yields $\frac{3}{2}\max\{2^{j\beta},2^{m\alpha}\}\le N_{j,m}\le 2\max\{2^{j\beta},2^{m\alpha}\}$).
 So
 $$
 \begin{aligned}
 &\ \  |TP_0^{(2)}f(x,y)|\\
 \lesssim&\ \sum_{\tau\in\Z}\frac{1}{1+|\tau|^8}
  \frac{1}{N_{j,m}}\sum_{r=0}^{N_{j,m}-1}\frac{1}{|I_r|}
  \int_{I_r}\int_\tau^{\tau+1}|f (x-t,y-z-u_k(x)[t]^\alpha-v_k(x)[t]^\beta)|dtdz.
 \end{aligned}
 $$
 Since  $(t,z)\in I_r\times [\tau,\tau+1]$, we have $z+u_k(x)[t]^\alpha+v_k(x)[t]^\beta\in J_r$, which is defined by
  $$
  \begin{aligned}
  J_r:=& [\tau+u_k(x)(\frac{1}{2}
  \lambda_{x,j}+r\delta_{x,j,m})^\alpha+v_k(x)(\frac{1}{2}
  \lambda_{x,j}+r\delta_{x,j,m})^\beta,\\
  &\ \tau+1+u_k(x)(\frac{1}{2}
  \lambda_{x,j}+(r+1)\delta_{x,j,m})^\alpha+v_k(x)(\frac{1}{2}
  \lambda_{x,j}+(r+1)\delta_{x,j,m})^\beta].
  \end{aligned}
  $$
  We also have
\begin{equation}\label{f1}
 \begin{aligned}
 &\ \ |TP_0^{(2)}f(x,y)|\\
 \lesssim&\ \sum_{\tau\in\Z}\frac{1}{1+|\tau|^8}
  \frac{1}{N_{j,m}}\sum_{r=0}^{N_{j,m}-1}\frac{1}{|I_r|}
  \int_{I_r}\int_{J_r}|f (x-t,y-z)|dtdz.
 \end{aligned}
 \end{equation}
  By the mean value theorem, we  see
  $$1\le |J_r|\lesssim\ 1+
 u_k(x)
\lambda_{x,j}^{\alpha-1}\delta_{x,j}
  + v_k(x)
  \lambda_{x,j}^{\beta-1}\delta_{x,j}.
  $$
 Recall $\delta_{x,j,m}=\lambda_{x,j}\min\{2^{-m\alpha},2^{-j\beta}\}$. It follows by (\ref{yong}) that
  \begin{equation}\label{goal}
  u_k(x)
\lambda_{x,j}^{\alpha-1}\delta_{x,j,m}
  + v_k(x)
  \lambda_{x,j}^{\beta-1}\delta_{x,j,m}\lesssim 1,
  \end{equation}
  which yields $1\le|J_r|\lesssim1$.
 Hence, there exists  a positive integer $C_{\alpha,\beta} (\lesssim1)$ such that
 \begin{equation}\label{f2}
  J_r\subset\big(\bigcup_{i=0}^{
  C_{\alpha,\beta}}[\tau+a_{j,m,k}(x)+i,\tau+a_{j,m,k}(x)+i+1] \big),
  \end{equation}
  where
  $$a_{j,m,k}(x):=u_k(x)(\frac{1}{2}
  \lambda_{x,j}+r\delta_{x,j,m})^\alpha+v_k(x)(\frac{1}{2}
  \lambda_{x,j}+r\delta_{x,j,m})^\beta.$$
  We first have by (\ref{yong}) that
  \begin{equation}\label{add1}
  a_{j,m,k}(x)\lesssim\ 2^{m\alpha}+2^{j\beta}.
  \end{equation}
  Furthermore,
  the definitions of $\lambda_{x,j}$, $\delta_{x,j,m}$ and $(u_k,v_k)$  gives
  $$
  \begin{aligned}
  a_{j,m,k}(x)=&\ 2^{j\beta}(\frac{1}{2}+r\min\{2^{-m\alpha},2^{-j\beta}\})^\beta\\
  &
  +2^{k(1-\frac{\alpha}{\beta})}u(x)v(x)^{-\frac{\alpha}{\beta}}
  2^{j\alpha}(\frac{1}{2}+r\min\{2^{-m\alpha},2^{-j\beta}\})^\alpha.
  \end{aligned}
  $$
  Combining with (\ref{f1}) and (\ref{f2}) leads to
$$
  \begin{aligned}
 &\ \ |T(P_0^{(2)}f)(x,y)|\\
 \lesssim&
 \ \sum_{\tau\in\Z}\frac{1}{1+|\tau|^8}
  \frac{1}{N_{j,m}}\sum_{r=0}^{N_{j,m}-1}
  \sum_{i=0}^{C_{\alpha,\beta}}\frac{1}{|I_r|}
  \int_{I_r}\int_{\tau+a_{j,m,k}(x)+i}^{\tau+a_{j,m,k}(x)+i+1}|f (x-t,y-z)|dtdz.
  \end{aligned}
  $$
By (\ref{add1}), we observe
  \begin{equation}\label{goal2}
  \frac{1}{|I_r|}
  \int_{I_r}\int_{\tau+a_{j,m,k}(x)+i}^{\tau+a_{j,m,k}(x)+i+1}|f (x-t,y-z)|dtdz
  \lesssim\ M_1^{[\sigma_r^1]}(\mathfrak{M}_2^{[\sigma_r^2(x,k,i)]}f)(x,y),
  \end{equation}
  where
  \begin{equation}\label{end1}
  \begin{aligned}
  &\ \sigma_r^1:=r+\max\{2^{m\alpha-1},2^{j\beta-1}\}\lesssim\ \max\{2^{j\beta},2^{m\alpha}\},\\
  &\sigma_r^2(x,k,i):=\tau+i+a_{j,m,k}(x)\lesssim\ \tau+C_{\alpha,\beta} +\max\{2^{j\beta},2^{m\alpha}\}.
  \end{aligned}
  \end{equation}
  Here $M^{[\cdot]}_i$ and $\mathfrak{M}^{[\cdot]}_i$ ($i=1,2$) are the shifted maximal operators (defined in Section \ref{s2}) applied in the $i-$th variable.
Notice  that $\sigma_r^2(x,k,i)$ depends on  $k$, which suggests that  our procedure  is  very different from \cite{GHLJ}, and more involved.
We now have
\begin{equation}\label{e1}
  \begin{aligned}
 |T(P_0^{(2)}f)(x,y)|\lesssim&
 \ \sum_{\tau\in\Z}\frac{1}{1+|\tau|^8}
  \frac{1}{N_{j,m}}\sum_{r=0}^{N_{j,m}-1}
  \sum_{i=0}^{C_{\alpha,\beta}}M_1^{[\sigma_r^1]}(
  \mathfrak{M}_2^{[\sigma_r^2(x,k,i)]}f)(x,y).
  \end{aligned}
  \end{equation}
  \begin{rem}\label{r10}
  We point out that we do not use $k<k_{\alpha,\beta,x}^1$ in the proof of  (\ref{e1}). In fact, (\ref{e1}) will be used in the proof of (\ref{a3}).
  \end{rem}
  We can not use the estimate of the shifted maximal operator given in Lemma \ref{sme} directly. However, we have assumed
 $\alpha<\beta$ and
$k\le k_{\alpha,\beta,x}^1=\frac{ \log_2 \big(v(x)^\alpha u(x)^{-\beta}\big)-\alpha\beta j}{\beta-\alpha}$ in (\ref{a1}) so that
$$2^{k(1-\frac{\alpha}{\beta})}u(x)v(x)^{-\frac{\alpha}{\beta}}
  2^{j\alpha}\le1.$$
 Thus, by Lemma \ref{lnew}, there exists a positive constant $C_{\alpha,\beta}'$ independent of $k$  such that
   $$M_1^{[\sigma_r^1]}(
  \mathfrak{M}_2^{[\sigma_r^2(x,k,i)]}f)(x,y)
  \le\ \sum_{l=0}^{C_{\alpha,\beta}'}M_1^{[\sigma_r^1]}(
  \mathfrak{M}_2^{[\sigma_r^3(i)+l]}f)(x,y),$$
  where
  $$\sigma_r^3(i):=\tau+i+2^{j\beta}
  (\frac{1}{2}+r\min\{2^{-m\alpha},2^{-j\beta}\})^\beta.$$
The scaling arguments  gives
\begin{equation*}
\begin{aligned}
&\ \ |T_{k,m,j}(P_k^{(2)}f)(x,y)|
=|T(P_0^{(2)}F_k)(x,2^ky)|\\
\lesssim&\ \sum_{\tau\in\Z}\frac{1}{1+|\tau|^8}
  \frac{1}{N_{j,m}}\sum_{r=0}^{N_{j,m}-1}
  \sum_{i=0}^{C_{\alpha,\beta}}\sum_{l=0}^{
  C_{\alpha,\beta}'}M_1^{[\sigma_r^1]
  }(\mathfrak{M}_2^{[\sigma_r^3(i)+l]}F_k)(x,2^ky)\\
  \lesssim&\ \sum_{\tau\in\Z}\frac{1}{1+|\tau|^8}
  \frac{1}{N_{j,m}}\sum_{r=0}^{N_{j,m}-1}
  \sum_{i=0}^{C_{\alpha,\beta}}\sum_{l=0}^{
  C_{\alpha,\beta}'}M_1^{[\sigma_r^1]
  }(M_2^{[\sigma_r^3(i)+l]}F_k)(x,2^ky)\\
\lesssim&\ \sum_{\tau\in\Z}\frac{1}{1+|\tau|^8}
  \frac{1}{N_{j,m}}\sum_{r=0}^{N_{j,m}-1}
  \sum_{i=0}^{C_{\alpha,\beta}}\sum_{l=0}^{C_{\alpha,\beta}'
  }M_1^{[\sigma_r^1]
  }(M_2^{[\sigma_r^3(i)+l]}f)(x,y).
\end{aligned}
\end{equation*}
Because of $P_k^{(2)}=\tilde{P}_k^{(2)}P_k^{(2)}$ (here $\tilde{P}_k^{(2)}=P_{k-1}^{(2)}+P_k^{(2)}+P_{k+1}^{(2)}$), we can obtain by the same way that
$$
\begin{aligned}
|T_{k,m,j}(P_k^{(2)}f)(x,y)|\lesssim&\
\sum_{\tau\in\Z}\frac{1}{1+|\tau|^8}
  \frac{1}{N_{j,m}}\sum_{r=0}^{N_{j,m}-1}\\
  &\
  \sum_{i=0}^{C_{\alpha,\beta}}\sum_{l=0}^{C_{\alpha,\beta}'
  }M_1^{[\sigma_r^1]
  }(M_2^{[\sigma_r^3(i)+l]}P_k^{(2)}f)(x,y).
  \end{aligned}
  $$
So we can accomplish (\ref{a1})  by Minkowski's inequality if we can show that
\begin{equation}\label{4.4}
\Big\|\big(\sum_{k\in\Z}|M_1^{[\sigma_r^1]
  }(M_2^{[\sigma_r^3(i)+l]}P_k^{(2)}f)(x,y)|^2\big)^\frac{1}{2}
\Big\|_p\lesssim_{\alpha,\beta}\ (1+\tau^2)(2+j^2\beta^2+m^2\alpha^2)^2\|f\|_p
\end{equation}
holds for all $0\le i\le C_{\alpha,\beta}$ and $0\le l\le C_{\alpha,\beta}'$, where the constant only  depends on $\alpha$ and $\beta$.
Thanks to Lemma \ref{sme} and (\ref{end1}),  the left-hand side of (\ref{4.4}) is controlled by
$$
\begin{aligned}
&\ \log^2(2+\sigma_r^1) \Big\|\big(\sum_{k\in \Z}|(M_2^{[\sigma_r^3(i)+l]}P_k^{(2)}f)(x,y)|^2\big)^\frac{1}{2}
\Big\|_p\\
 \lesssim&\
(2+j^2\beta^2+m^2\alpha^2) \Big\|\big(\sum_{k\in\Z}|(M_2^{[\sigma_r^3(i)+l]}P_k^{(2)}f)(x,y)
|^2\big)^\frac{1}{2}
\Big\|_p\\
\lesssim&\  (2+j^2\beta^2+m^2\alpha^2)\Big\| \log^2(2+\sigma_r^3(i)+l)\|\big(\sum_{k\in\Z}
|(P_k^{(2)}f)(x,y)
|^2\big)^\frac{1}{2}\|_{L^p_y}\Big\|_{L^p_x} \\
\lesssim&\ \log^4(2+|\tau|)(2+j^2\beta^2+m^2\alpha^2)^2\|f\|_p,
\end{aligned}
$$
 which is bounded by the right-hand of (\ref{4.4}).
\vskip.3in
{\bf $\bullet$ The proof of (\ref{a2})}
The proof is very similar to the proof of (\ref{a1}). Indeed, taking
$$\lambda_{x,m}=2^mu_k(x)^{-\frac{1}{\alpha}},
\ \ \delta_{x,j,m}=\lambda_{x,m}\min\{2^{-m\alpha},2^{-j\beta}\},$$
where $\lambda_{x,m}$ plays the same role as $\lambda_{x,j}$ (see (\ref{yong})), and repeating the previous arguments yields
$$
  \begin{aligned}
 |T(P_0^{(2)}f)(x,y)|\lesssim&
 \ \sum_{\tau\in\Z}\frac{1}{1+|\tau|^8}
  \frac{1}{N_{j,m}}\\
  &\ \sum_{r=0}^{N_{j,m}-1}
  \sum_{i=0}^{C_{\alpha,\beta}}\frac{1}{|I_r|}
  \int_{I_r}\int_{\tau+a_{j,m,k}+i}^{\tau+a_{j,m,k}+i+1}|f (x-t,y-z)|dtdz,
  \end{aligned}
  $$
  where $N_{j,m}$, $I_r$ is defined as before  and
  $$a_{j,m,k}(x):=u_k(x)(\frac{1}{2}
  \lambda_{x,m}+r\delta_{x,j,m})^\alpha+v_k(x)(\frac{1}{2}
  \lambda_{x,m}+r\delta_{x,j,m})^\beta$$
  satisfying
  $$
  \begin{aligned}
  a_{j,m,k}(x)=& 2^{m\alpha}(\frac{1}{2}+r\min\{2^{-m\alpha},2^{-j\beta}\})^\alpha\\
  &
  +2^{k(1-\frac{\beta}{\alpha})}v(x)u(x)^{-\frac{\beta}{\alpha}}
  2^{m\beta}(1+r\min\{2^{-m\alpha},2^{-j\beta}\})^\beta
  \end{aligned}
  $$
  and
  $$a_{j,m,k}(x)\lesssim\ 2^{m\alpha}+2^{j\beta}.$$
Since we have assumed that $\alpha<\beta$ and
$k\ge k_{\alpha,\beta,x}^2=\frac{\alpha\beta m+ \log_2 \big(v(x)^\alpha u(x)^{-\beta}\big)}{\beta-\alpha}$,
we have
$$2^{k(1-\frac{\beta}{\alpha})}v(x)u(x)^{-\frac{\beta}{\alpha}}
  2^{m\beta}\le 1.$$
Then along the previous way leads to (\ref{a2}).
\vskip.1in
{\bf $\bullet$ The proof of (\ref{a3})}
Thanks to (\ref{jia1}),
\begin{equation}\label{ssl}
\sharp\{k\in\Z:\ k_{\alpha,\beta,A}^1\le k< k_{\alpha,\beta,A}^2\}\lesssim_{\alpha,\beta}(j\beta+m\alpha).
\end{equation}
By Minkowski's inequality and (\ref{ssl}),  the left hand side of (\ref{a3}) is bounded by
$$
\begin{aligned}
&\Big\|\sum_{ k_{\alpha,\beta,A}^1\le k\le k_{\alpha,\beta,A}^2 }|T_{k,m,j}(P_k^{(2)}f)|
 \Big\|_p\\
 \lesssim&\ \sharp\{k\in\Z:\ k_{\alpha,\beta,A}^1\le k< k_{\alpha,\beta,A}^2\}\sup_{k\in\Z}\|T_{k,m,j}(P_k^{(2)}f)(x,y)\|_p\\
\lesssim&\ (j\beta+m\alpha)
\sup_{k\in\Z}\|T_{k,m,j}(P_k^{(2)}f)(x,y)\|_p.
\end{aligned}
$$
Recall $F_k(x,y)=f(x,2^{-k}y)$. As the statement in Remark \ref{r10}, thanks to (\ref{e1}), we deduce by the scaling arguments that
\begin{equation*}
  \begin{aligned}
 |T_{k,m,j}&(P_k^{(2)}f)(x,y)|
=|T(P_0^{(2)}F_k)(x,2^ky)|\\
\lesssim&\ \sum_{\tau\in\Z}\frac{1}{1+|\tau|^8}
  \frac{1}{N_{j,m}}\sum_{r=0}^{N_{j,m}-1}
  \sum_{i=0}^{C_{\alpha,\beta}}M_1^{[\sigma_r^1]
  }(M_2^{[\sigma_r^2(x,k,i)]}F_k)(x,2^ky)\\
  \lesssim&\ \sum_{\tau\in\Z}\frac{1}{1+|\tau|^8}
  \frac{1}{N_{j,m}}\sum_{r=0}^{N_{j,m}-1}
  \sum_{i=0}^{C_{\alpha,\beta}}M_1^{[\sigma_r^1]
  }(M_2^{[\sigma_r^2(x,k,i)]}f)(x,y).
  \end{aligned}
  \end{equation*}
We complete the proof of (\ref{a3}) by  (\ref{end1}) and the scalar shifted maximal estimate (see (\ref{ssme1}) in Lemma \ref{ssme}).
\subsection{The estimates of $H_{\alpha,\beta}^{lh}P_k^{(2)}f$ and $H_{\alpha,\beta}^{hl}P_k^{(2)}f$}
Due to the property of the support of $\phi$, $\Gamma(x,t)$  essentially plays the same role as  the single fractional monomial. So we expect that the proof is similar to the previous work \cite{GHLJ}, and easier than the  estimate of $H_{\alpha,\beta}^{hh}P_k^{(2)}f$.
\vskip.1in
We only give a sketch of the estimate of $H_{\alpha,\beta}^{lh}P_k^{(2)}f$. By the variable substitution $l\rightarrow -k/\beta+j$, we get
$$H_{\alpha,\beta}^{lh}P_k^{(2)}f(x,y)=\sum_{j\ge 0}T_{k,j}(P_k^{(2)}f)(x,y),$$
where $T_{k,j}$ is defined by
 $$T_{k,j} f(x,y)
 =\int_\R f(x-t,y-u(x)[t]^\alpha-v(x)[t]^\beta)
\phi_{-\frac{k}{\alpha}}(u(x)^\frac{1}{\alpha}t)
\psi_{j-\frac{k}{\beta}}(v(x)^\frac{1}{\beta}t)\frac{dt}{t}.$$
It suffices to show  there exists $\iota>0$ such that
 \begin{equation*}
 \Big\|\big(\sum_{k\in\Z}|T_{k,j}(P_k^{(2)}f)|^2
 \big)^\frac{1}{2}\Big\|_p
 \lesssim\ 2^{-j\iota}\|f\|_p.
 \end{equation*}
 This, with the $L^2$ bound (which can be obtained by the same way leading to Lemma \ref{l2}), yields that it is enough to show
 $$\Big\|\big(\sum_{k\in\Z}|T_{k,j}(P_k^{(2)}f)|^2
 \big)^\frac{1}{2}\Big\|_p
 \lesssim\ (1+j)^4\|f\|_p.$$
 Since $t$ is in the support of $\phi_{-\frac{k}{\alpha}}(u(x)^\frac{1}{\alpha}t)$, after using the scaling arguments,
  $u(x)[t]^\alpha+v(x)[t]^\beta$ is approximated by $v(x)[t]^\beta$.  Indeed, along the same way yielding  (\ref{e1}), if we use the following notation:
  $$\lambda_{x,j}:=2^jv_k(x)^{-\frac{1}{\beta}},
  \ \ \delta_{x,j}=\lambda_{x,j}2^{-j\beta},\ N_{jm}\delta_{x,j}
  \in [\frac{3}{2}\lambda_{x,j},2\lambda_{x,j}],\ $$
  one can obtain
 $$T_{k,j}(P_k^{(2)}f)
 \lesssim\ \sum_{\tau\in\Z}\frac{1}{1+|\tau|^8}
  \frac{1}{N_{j,m}}\sum_{r=0}^{N_{j,m}-1}
  \sum_{i=0}^{C_{\alpha,\beta}}M_1^{[n_1(r)]
  }(M_2^{[n_2(r)+i+\tau]}P_k^{(2)}f)(x,y),
 $$
 where
 $$n_1(r)=r+2^{j\beta-1}\lesssim 2^{j\beta},\ n_2(r)+i+\tau\lesssim 2^{j\beta}+\tau.$$
 We complete the proof by applying Lemma \ref{sme} directly.
 \vskip.3in
 \section{Proof of Lemma \ref{l4.1}}
 \label{s5}
 \vskip .1in
 This section devotes to the proof of Lemma \ref{l4.1}.
 Since this estimate is independent of $u(x)$, $v(x)$ and $k$, by the scaling arguments, it suffices to show the case $k=0$:
 \begin{equation}\label{low1}
 \begin{aligned}
 H_{\alpha,\beta}^{ll}P_0^{(2)}f(x,y)
\lesssim&\ \sum_{q\in\Z}\frac{1}{(1+|q|)^2}
\int_q^{q+1} \mathcal{M}^{(1)}P_0^{(2)}f(x,y-z)  dz\\
&+\mathcal{M}^{(1)}P_0^{(2)}f(x,y)+H^\star P_0^{(2)}f (x,y),
\end{aligned}
\end{equation}
where
$$H_{\alpha,\beta}^{ll}P_0^{(2)}f(x,y)
:=p.v.\int_{\R} (P_0^{(2)}f)(x-t,y-\Gamma(x,t))  \phi(u(x)^\frac{1}{\alpha}t)\phi(v(x)^\frac{1}{\beta}t)\frac{dt}{t}.$$
Due to the support of $\phi$ yielding
\begin{equation}\label{3.0}
\big|\Gamma(x,t)\big|\le\ u(x)|t|^\alpha+v(x)|t|^\beta\lesssim 1,
\end{equation}
 $H_{\alpha,\beta}^{ll}$ can be approximated by $T^{ll}_{\alpha,\beta}$, which is defined by
$$T^{ll}_{\alpha,\beta}f(x,y):=p.v.
\int_\R f(x-t,y)\phi(u(x)^\frac{1}{\alpha}t)
\phi(v(x)^\frac{1}{\beta}t)\frac{dt}{t}.$$
To obtain (\ref{low1}), it suffices to show
\begin{equation}\label{3.3}
|H_{\alpha,\beta}^{ll}P_0^{(2)}f-T^{ll}_{\alpha,\beta}P_0^{(2)}f
|\lesssim\ \sum_{q\in\Z}\frac{1}{(1+|q|)^2}
\int_q^{q+1} \mathcal{M}^{(1)}P_0^{(2)}f(x,y-z)  dz
\end{equation}
and
\begin{equation}\label{3.2}
|T^{ll}_{\alpha,\beta} P_0^{(2)}f|\lesssim\ \mathcal{M}^{(1)}P_0^{(2)}f(x,y)+H^\star P_0^{(2)}f (x,y).
\end{equation}
We first prove (\ref{3.3}).
Denote the Schwartz function $\Phi(z)$ by
$$\sum_{i=-1}^1P_i^{(2)} f(x,y)= \int f(x,y-z)\Phi(z)dz,$$
which, together with $P_0^{(2)}=\sum_{i=-1}^1P_i^{(2)}P_0^{(2)}$, yields
$$P_0^{(2)}f(x,y)=\sum_{i=-1}^1P_i^{(2)}P_0^{(2)} f=\int_\R P_0^{(2)}f(x,y-z)\Phi(z)dz.$$
Then we rewrite
$H_{\alpha,\beta}^{ll}P_0^{(2)}f(x,y)-T^{ll}_{\alpha,\beta}
P_0^{(2)}f(x,y)$
 as
\begin{equation}\label{3.1}
\int\int (P_0^{(2)}f)(x-t,y-z)
\Big\{\Phi(z-\Gamma(x,t))-
\Phi(z)\Big\}\phi(u(x)^\frac{1}{\alpha}t)
\phi(v(x)^\frac{1}{\beta}t)\frac{dt}{t}dz.
\end{equation}
By the mean value theorem  and (\ref{3.0}),  we deduce
$$
\begin{aligned}
\big|\Phi(z-\Gamma(x,t))-
\Phi(z)\big|
=&\ \left|\int_0^1 \frac{d}{ds}(\Phi(z-s\Gamma(x,t))) ds\right|\\
\le&\ |\int_0^1 \Phi'(z-s\Gamma(x,t)) ds|\ |\Gamma(x,t)|\\
\lesssim&\ \int_0^1 \frac{ds}{(1+|z-s\Gamma(x,t)|)^2}\  |\Gamma(x,t)|\\
\lesssim&\  \sum_{q\in\Z}
\frac{1}{(1+|q|)^2}\ \chi_{[q,q+1)}(z)\ \big(u(x)
|t|^\alpha+v(x)|t|^\beta\big).
\end{aligned}
$$
With this we bound
(\ref{3.1})  by
$$
\begin{aligned}
&\ \ \sum_{q\in\Z}\frac{1}{(1+|q|)^2}
\int_q^{q+1}\int |(P_0^{(2)}f)(x-t,y-z)|\\
\times&
\big(u(x)|t|^{\alpha-1}\phi(u(x)^\frac{1}{\alpha}t)
+v(x)|t|^{\beta-1}\phi(v(x)^\frac{1}{\beta}t)\big)
dt dz.
\end{aligned}
$$
Here $|\phi|\le1$ is applied.
Thus, to prove (\ref{3.3}),  it suffices to show
\begin{equation}\label{jida}
\begin{aligned}
&\ \ \int |(P_0^{(2)}f)(x-t,y)|
\big(u(x)|t|^{\alpha-1}\phi(u(x)^\frac{1}{\alpha}t)
+v(x)|t|^{\beta-1}\phi(v(x)^\frac{1}{\beta}t)\big)
dt\\
\lesssim&\  \mathcal{M}^{(1)}(P_0^{(2)}f)(x,y),
\end{aligned}
\end{equation}
the left side of which   is  the sum of
$$\int |(P_0^{(2)}f)(x-t,y)|
u(x)|t|^{\alpha-1}\phi(u(x)^\frac{1}{\alpha}t)
dt$$
and
$$\int |(P_0^{(2)}f)(x-t,y)|v(x)|t|^{\beta-1}\phi(v(x)^\frac{1}{\beta}t)dt.$$
We only show the estimate of the former  since that of the  latter is similar.
If $\alpha>1$, we apply
$u(x)|t|^{\alpha}\le 2^{\alpha}$ derived from the support of $\phi(u(x)^\frac{1}{\alpha}t)$ to obtain
$$
\begin{aligned}
&\ \ \int |(P_0^{(2)}f)(x-t,y)|
u(x)|t|^{\alpha-1}\phi(u(x)^\frac{1}{\alpha}t)
dt\\
\lesssim_\alpha&\ u(x)^{\frac{1}{\alpha}}\int_{|t|\le 2u(x)^{-\frac{1}{\alpha}}} |(P_0^{(2)}f)(x-t,y)|
dt,
\end{aligned}
$$
which is bounded by a constant multiple of $\mathcal{M}^{(1)}(P_0^{(2)}f)(x,y)$.
If $0<\alpha<1$,
we have
\begin{equation}\label{3.00}
\int |(P_0^{(2)}f)(x-t,y)|
u(x)|t|^{\alpha-1}\phi(u(x)^\frac{1}{\alpha}t)
dt
\le\ \sup_{\epsilon>0}\big\{\int |P_0^{(2)}f(x-t,y)| \epsilon^{-1} K(\epsilon^{-1}t)dt\big\}
\end{equation}
where
 $K(t)=|t|^{\alpha-1}\phi(t).$
 It follows from $0<\alpha<1$ and the property of $\phi$ that
 $$\|K(t)\|_1\lesssim_\alpha 1,\ \ K(t)=K(|t|),\ \ K(t_1)\le K(t_2)$$
whenever $|t_1|\ge |t_2|$.  By  Theorem 2.1.10 in \cite{G08},  the right side  of (\ref{3.00}) is bounded by a constant multiple of $\mathcal{M}^{(1)}P_0^{(2)}f(x,y)$, which  completes the proof of (\ref{3.3}).
 \vskip.1in
 Next, we show (\ref{3.2}). Because of the support  of $\phi$, the region of integration is $I:=\{|t|\le \sigma(x)\}$, where $\sigma(x)=\min\{2u(x)^{-\frac{1}{\alpha}},
 2v(x)^{-\frac{1}{\beta}}\}$. A simple computation gives
 $$
 \begin{aligned}
 |T^{ll}_{\alpha,\beta}P_0^{(2)}f(x,y)|\le&\ \left|p.v.\int_{I} P_0^{(2)}f(x-t,y)\frac{dt}{t}\right|\\
 &+\left|p.v.\int_{I} P_0^{(2)}f(x-t,y) \frac{\phi(u(x)^\frac{1}{\alpha}t)-1}{t}
 \phi(v(x)^\frac{1}{\beta}t)dt\right|\\
 &+\left|p.v.\int_{I} P_0^{(2)}f(x-t,y) \frac{\phi(v(x)^\frac{1}{\beta}t)-1}{t}
 dt\right|.
 \end{aligned}
 $$
We bound the first term  by
 $H^\star P_0^{(2)}f(x,y)$, and the remainder  by $C\mathcal{M}^{(1)}P_0^{(2)}f(x,y)$. As a consequence,  $|T^{ll}_{\alpha,\beta}P_0^{(2)}f(x,y)|$ is
  bounded by a constant multiple of
 $$H^\star P_0^{(2)}f(x,y)+\mathcal{M}^{(1)}P_0^{(2)}f(x,y).$$
  This concludes the proof of (\ref{3.2}).
  \vskip.3in
\section{Proof of Lemma \ref{l2}}
\label{s6}
\vskip .1in
We prove Lemma \ref{l2} in section \ref{s4}.
 The advantage for the case $p=2$ is that  Minkowski's inequality plays a positive role.
By Littlewood-Paley theory, (\ref{4.2}) for $p=2$ is a direct result of
\begin{equation}\label{d1}
 \|T_{k,m,j}(P_k^{(2)}f)\|_2\lesssim\ 2^{-\max\{j\beta,m\alpha\}\tau_2} \|P_k^{(2)}f\|_2
 \end{equation}
 for some $\tau_2>0$. By the scaling arguments, (\ref{d1}) is equivalent to   the case $k=0$, that is,
\begin{equation}\label{100}
\begin{aligned}
\|\int_{\R} (P_0^{(2)}f)(x-t,y-\Gamma(x,t))&  \psi_l (u(x)^\frac{1}{\alpha}t)
\psi_m(v(x)^\frac{1}{\beta}t)\frac{dt}{t}\|_2\\
\lesssim&\ 2^{-\gamma_2'\max\{\beta m,\alpha l\}}\|P_0^{(2)}f\|_2
\end{aligned}
\end{equation}
holds for some $\gamma_2'>0$.
Applying Plancherel's  theorem in the $y$ variable again, it is enough to show
\begin{equation}\label{decay2}
\begin{aligned}
\|\int g(x-t,\eta)e^{-i\big(u(x)[t]^\alpha+v(x)[t]^\beta
\big)\eta}&\psi_l(u(x)^\frac{1}{\alpha}t)
\psi_m(v(x)^\frac{1}{\beta}t)\frac{dt}{t}\|_{L^2_\eta(L^2_x)}\\
\lesssim&\ 2^{-\gamma_2'\max\{\beta m,\alpha l\}}\|g\|_2.
\end{aligned}
\end{equation}
We bound the integral on the left side  by
$(\digamma\star \Phi_a^{\lambda,b})(x,\eta)$  with
$$
\digamma=g(x,\eta)=\mathcal{F}^y(P_0^{(2)}f)(x,\eta),\ a=2^lu(x)^{-\frac{1}{\alpha}},\
b=2^{-m}v(x)^\frac{1}{\beta}a,\
\ \lambda_1=2^{l\alpha}\eta,\ \lambda_2=\frac{v(x)}{2^{-l\beta}u(x)^\frac{\beta}{\alpha}}\eta.
$$
Then Lemma \ref{l2} gives (\ref{decay2}) with $2^{-\gamma_2'\max\{\beta m,\alpha l\}}$ replaced by  $2^{-\gamma_0\alpha l }$. Similarly,
the integral on the left side  can also be bounded by
$(\digamma\star \Phi_a^{\lambda,b})(x,\eta)$  with
$$
\digamma=g(x,\eta)=\mathcal{F}^y(P_0^{(2)}f)(x,\eta),\ a=2^mv(x)^{-\frac{1}{\beta}},\
b=2^{-l}u(x)^\frac{1}{\alpha}a,\
\ \lambda_1=2^{m\beta}\eta,\ \lambda_2=\frac{u(x)}{2^{-m\alpha}v(x)^\frac{\alpha}{\beta}}\eta.
$$
By Lemma \ref{l2}, we  obtain  (\ref{decay2}) with $2^{-\gamma_2'\max\{\beta m,\alpha l\}}$ replaced by  $2^{-\gamma_0\beta m }$. Combing with the above two bound yields  (\ref{decay2}) for $\gamma_2'=\gamma_0$.
\vskip.1in
 \vskip.2in
\section*{Acknowledgements}
\vskip .1in
The author would like to thank Prof. Jiecheng Chen for his encouragement.
 This work was supported by the NSF of China 11901301, the NSF of Jiangsu Province BK20180721,  the NSF of the Jiangsu Higher Education Institutions of China
(18KJB110018),.
\vskip.3in

\end{document}